\documentclass[a4paper,12pt]{article}
\usepackage[english]{babel}
\usepackage{hyperref}
\usepackage{indentfirst}
\usepackage{amsfonts,amssymb,amsthm,amsmath}
\theoremstyle{definition}
\newtheorem{De}{Definition}
\theoremstyle{theorem}
\newtheorem{Th}{Theorem}
\theoremstyle{lemma}
\newtheorem{Lm}{Lemma}
\theoremstyle{plain}

\theoremstyle{remark}

\theoremstyle{plain}

\newcommand{\ir}{[0,T]\times \mathbb{R}^{n}}

\begin{document}
\title{Inverse Stackelberg Solutions for  Games with Many Followers}
\author{Yurii Averboukh\footnote{Krasovskii Institute of Mathematics and Mechanics UrB RAS, S.~Kovalevskaya str. 16, Ekaterinburg, 620990, Russia, e-mail: ayv@imm.uran.ru, averboukh@gmail.com.}}
\date{}
\maketitle
\begin{abstract}
The paper is devoted to inverse Stackelberg games with many players. We consider both static and differential games. The main assumption of the paper is the compactness of the strategy sets. We obtain the characterization of inverse Stackelberg solutions and under additional concavity conditions establish the existence theorem.
\end{abstract}
{\small \textbf{Keywords:} {Inverse Stackelberg games; incentives; differential games.}

\vspace{6pt}\noindent\textbf{AMS 2010 Subject Classification.} 91A10, 91A06, 91A23, 49N70.}

\section{Introduction}

The paper is devoted to the inverse Stackelberg games, also known as incentive problems.
In the ordinary Stackelberg games one player (called a leader) announces his strategy when the other players (called followers) maximize their payoffs using this information. In the inverse Stackelberg games the leader announces the incentive strategy i.e. the reaction to the followers' strategies (see \cite{Ho_2}, \cite{Ho_1}, \cite{Ho_3}, \cite{olsder1}, \cite{olsder2} and reference therein).  For dynamic case the reaction should be nonanticipative.

The inverse Stackelberg games appear in several models (see for example \cite{allocation}, \cite{Stankova}). In the games with many followers it is often assumed that followers play a Nash game (see \cite{Basar}, \cite{allocation}, \cite{examplnashfoll}).  If the strategy sets are normed space then the incentive strategy can be constructed in the affine form (see \cite{Zheng_Basar_affine} for static games, and \cite{Ehtama_Hamalainen} for differential games).

In this paper we consider the static and differential games with many follower.
The main assumption in the paper is the compactness of the strategy sets. In this case the most efficient tool is discontinuous incentive strategies realizing the concept of punishment. \cite{Cleimenov} first applied punishment strategies to the feedback differential Stackelberg games. The inverse Stackelberg solutions of two-person differential games were studied via punishment strategies in the paper by \cite{averboukhbaklanov}. In that paper the authors described the set of inverse Stackelberg solutions  and showed its nonemptiness. In particular, the set of inverse Stackelberg payoffs is equal to the set of feedback Stackelberg payoffs. Note  that the incentive strategies considered in the paper by \cite{averboukhbaklanov} use full memory, i.e. the leader plays with the nonanticipating strategies proposed in the papers by \cite{elliotkalyon}, and \cite{VaraiyaLin} for zero-sum differential games. The usage of the strategies depending only on the current follower's control decreases the payoffs.

In this paper  punishment strategies are applied to the static inverse Stackelberg games and to the differential inverse Stackelberg games with many follower. We obtain the characterization of inverse Stackelberg solution and under additional concavity conditions establish the existence theorem.

The paper is organized as follows. Section \ref{sec_static} starts with the two-player static inverse Stackelberg game. Here there exists only one follower. We give the characterization of the solutions in this case and compare it with the ordinary Stackelberg solutions. Then we consider the static inverse Stackelberg game for the case of $n$ followers. The differential game case is considered in Section \ref{sec_diff}. In Section \ref{sec_ex_diff} we prove the existence theorem for the inverse Stackelberg solution of differential game.
\section{Static Games}\label{sec_static}
\subsection{Inverse Stackelberg Solutions for Two-player games}
We assume that the set of the players is $\{0,1\}$. Let $P_i$ be a set of strategies of player $i$; and let  $J_i(u_1,u_2)$ be an utility (payoff) function for player $i$. We assume that the sets $P_i$ are compact, and the functions $J_i$ are continuous. Each player wants to maximize his payoffs.

For definiteness let player $0$ be a leader, and let player $1$ be a follower.
In the inverse Stackelberg game the leader uses an incentive strategy $\alpha[u_1]$.  Here $\alpha[\cdot]$ is an arbitrary map from $P_1$ to $P_0$. The information about chosen incentive strategy of the leader is known to the follower.

Let $\alpha$ be a leader's incentive strategy. We say that $u_1^*$ is an optimal strategy of the follower if $$J_1(\alpha[u_1],u_1)\leq J_1(\alpha[u_1^*],u_1^*).$$ Denote the set of optimal  strategies of the follower by $\mathcal{F}(\alpha)$.

\begin{De}The pair consisting of incentive strategy of the leader $\alpha^*$ and the strategy of the follower $u_1^*$ is said to be an inverse Stackelberg solution if
\begin{enumerate}
  \item $u_1^*\in\mathcal{F}(\alpha^*)$;
  \item for any incentive strategy of the leader $\alpha$ the following inequality holds $$J_0(\alpha^*[u_1^*],u_1^*)\geq\max\{J_0(\alpha[u_1],u_1):u_1\in \mathcal{F}(\alpha)\}. $$
\end{enumerate}
\end{De} The second conditions in particular means that we consider the team solution.

The inverse Stackelberg solution  can be described by means of the lower value of the auxiliary zero-sum game in which  player $1$ wishes to maximize his payoff
\begin{equation}\label{V_def}
V^-=\max_{u_1\in P_1}\min_{u_0\in P_0}J_1(u_0,u_1).
\end{equation}

Let $\mathcal{A}$ be a set of pairs of strategies $(u_0,u_1)$ such that $J_1(u_0,u_1)\geq V^-$.

\begin{Lm}\label{lm_optimal}
If $u_1^\natural\in\mathcal{F}(\alpha)$, then $(\alpha[u_1^\natural],u_1^\natural)\in\mathcal{A}$.
\end{Lm}
\begin{proof}
Let $\hat{u}_1$ maximize the right-hand side of (\ref{V_def}). We have that for all $u_0\in P_0$
$J_1(u_0,\hat{u}_1)\geq V^- $. Therefore, $$J_1(\alpha[u_1^\natural],u_1^\natural)\geq J_1(\alpha[\hat{u}_1],u_1)\geq V^-. $$
\end{proof}
The converse statement is also true.
\begin{Lm}\label{lm_constr} Let $(u_0^\natural,u_1^\natural)\in\mathcal{A}$, then there exists an incentive strategy of the leader $\alpha$ such that
$u_0^\natural=\alpha[u_1^\natural]$ and $u_1^\natural\in\mathcal{F}[\alpha]$.
\end{Lm}
\begin{proof} For $u_1\in P_1$ let $\beta[u_1]\in{\rm Argmin}\{J_1(u_0,u_1):u_0\in P_0\}$.
Put $$\alpha[u_1]=\left\{
\begin{array}{cc}
  u_0^\natural, & u_1=u_1^\natural \\
  \beta[u_1], & u_1\neq u_1^\natural
\end{array}
\right.
$$
If $u_1\neq u_1^\natural$, then $$J_1(\alpha[u_1],u_1)=J_1(\beta(u_1),u_1)\leq V^-\leq J_1(u_0^\natural,u_1^\natural)=J_2(\alpha[u_1^\natural],u_1^\natural). $$
Therefore, $u_1^\natural\in\mathcal{F}(\alpha)$.
\end{proof}
The definition of inverse Stackelberg solution and lemmas 1, 2 yield the following Theorem.
\begin{Th}\label{th_inverse_static}
\begin{enumerate}
  \item If $(\alpha^*,u_1^*)$ is an inverse Stackelberg solution, then the pair $(u_0^*,u_1^*)$  with $u_0^*=\alpha^*[u_1^*]$ belongs to the set $\mathcal{A}$  and
      \begin{equation}\label{u_1_max_A}
      (u_0^*,u_1^*)\in {\rm Argmax}\{J_0(u_0,u_1):(u_0,u_1)\in\mathcal{A}\}.
      \end{equation}
  \item If the pair $(u_0^*,u_1^*)\in\mathcal{A}$ satisfies  condition (\ref{u_1_max_A}), then there exists an incentive strategy of the leader $\alpha^*$ such that $u_0^*=\alpha^*[u_1^*]$ and $(\alpha^*,u_1^*)$ is an incentive Stackelberg solution.
  \item There exists at least one inverse Stackelberg solution.
\end{enumerate}
\end{Th}
\begin{proof}
The first two statements directly follow from the definition of inverse Stackelberg solution and lemmas 1, 2.

The third statement follows from the second one and the compactness of $\mathcal{A}$.
\end{proof}

Now let us compare the payoffs given by inverse and ordinary Stackelberg solutions.
Recall the definition of the Stackelberg solution. Let $F(u_0)$ be a set of strategies $u_1^\natural$ such that $u_1^\natural$ maximizes the function $u_1\mapsto J_1(u_0,u_1)$. The pair $(u_0^\dag,u_1^\dag)$ is said to be a Stackelberg solution if
\begin{itemize}
  \item $u_1^\dag\in F(u_0^\dag)$;
  \item $J_0(u_0^\dag,u_1^\dag)=\max\{J_0(u_0,u_2):u_0\in F(u_0)\}$.
\end{itemize}
Note that if $(u_0^\dag,u_1^\dag)$ is the Stackelberg solution then
$$J_1(u_0^\dag,u_1^\dag)\geq V^+=\min_{u_0\in P_0}\max_{u_1\in P_1}J_1(u_0,u_1). $$ Here $V^+$ is the upper value of the auxiliary zero-sum game; $V^-\leq V^+$.
If $(u_0^\dag,u_1^\dag)$ is a Stackelberg solution, and $(\alpha^*,u_1^*)$ is an inverse Stackelberg solution, then
\begin{equation}\label{J_1_stack_inverse}
J_0(u_0^\dag,u_1^\dag)\leq J_0(\alpha^*[u_1^*],u_1^*).
\end{equation}
Indeed, denote $u_0^*=\alpha^*[u_1^*]$. By Theorem \ref{th_inverse_static} we have that $(u_0^*,u_1^*)$ maximizes the value of $J_0$ over the set $\mathcal{A}$. The pair $(u_0^\dag,u_1^\dag)$ maximizes the value of $J_0$ over the set $\{(u_0,u_1):u_1\in F(u_0)\}$. Inequality (\ref{J_1_stack_inverse}) follows from this and the inclusion $$\{(u_0,u_1):u_1\in F(u_0)\}\subset\{(u_0,u_1)\geq V^+\}\subset \mathcal{A}.$$

The following example shows that the inequality in (\ref{J_1_stack_inverse}) can be strick even in the case when $V^-=V^+$.
Let $J_0=u_0-u_1$, $J_1=u_0+u_1$, $u_0,u_1\in [-1,1]$. We have that $V^-=V^+=0$. The Stackelberg solution is the pair $(1,1)$; $J_1(1,1)=0$, $J_2(1,1)=2$.

Note that the pair $(1,-1)$ maximizes the value of $J_0$ over the set $\{(u_0,u_1):u_0-u_1\geq 0\}$. The inverse Stackelberg solution is the pair $(\alpha^*,-1)$ with $$\alpha^*_0[u_1]=\left\{
\begin{array}{cc}
  1, & u_1=-1, \\
  -1, & u_1\neq -1.
\end{array}
\right. $$
Consequently, we have that in this example the inverse Stackelberg solution gives a larger payoff than the Stackelberg solution
$J_0(\alpha^*_0[-1],-1)=2>0=J_0(1,1)$.
\subsection{Case of One Leader and Many Followers}
Let  player $0$ be a leader, and let players $1,\ldots n$ be followers. Player $i$ has a set of strategies $P_i$ and a payoff function $J_i$. As above, we assume that the sets $P_i$ are compact, the functions $J_i$ are continuous.

The incentive strategy of the leader is a mapping $$\alpha:\times_{i=1}^nP_i\rightarrow P_0.$$ To define the inverse Stackelberg game we should specify the solution concept used by  followers. We suppose that the followers play Nash game. Let $$P=\times_{i=1}^n P_i. $$ An element $u=(u_1,\ldots,u_n)$ of $P$ is a profile of followers' strategies. If $u_i'\in P_i$ then $(u_i',u_{-i})$ is the profile of strategies $(u_1,\ldots,u_{i-1},u_i',u_{i+1},\ldots,u_n)$. For simplification we write $J_i(u_0,u)$ to denote $J_i(u_0,u_1,\ldots,u_n)$. Analogously, put $J_i(u_0,u_i',u_{-i})\triangleq J_i(u_0,(u_i',u_{-i}))$. If $\alpha$ is an incentive strategy of the leader, $u$ is a profile of strategies of the followers, then denote $J_i[\alpha,u]=J_i(\alpha[u],u)$,  $J_i[\alpha,u_i',u_{-i}]=J_i[\alpha,(u_i',u_{-i})]$.
Further, let $\mathcal{E}(\alpha)$ be a set of followers' Nash equilibria in the case when the leader play with the incentive strategy $\alpha$:
$$\mathcal{E}(\alpha)=\{u:J_i[\alpha,u]\leq J_i[\alpha,u_i',u_{-i}], \ \ i=\overline{1,n},\ \ u_i'\in P_i\}. $$
\begin{De}
The pair $(\alpha^*,u^*)$ is an inverse Stackelberg solution in the game with one leader and $n$ followers playing Nash equilibrium if
\begin{enumerate}
  \item $u^*\in\mathcal{E}(\alpha)$.
  \item $$J_0[\alpha^*,u^*]=\max_{\alpha}\max_{u\in\mathcal{E}(\alpha)}J_0[\alpha,u]. $$
\end{enumerate}
\end{De}
The structure of inverse Stackelberg solution is given in the following statements. Denote
$$\mathcal{B}=\left\{(u_0^\natural,u^\natural):J_i(u_0^\natural,u^\natural)\geq \max_{u_i}\min_{u_0}J_i(u_0,u_i,u^\natural_{-i}), \ \ i=\overline{1,n}\right\}. $$
\begin{Lm}\label{lm_const_opt_many}
\begin{enumerate}
  \item If $u^\natural\in\mathcal{E}(\alpha)$, then $(\alpha[u^\natural],u^\natural)\in\mathcal{B}$;
  \item If the strategy of the leader $u_0^\natural$, and the profile of the followers' strategies $u^\natural$ are so that $(u^\natural,u^\natural)\in\mathcal{B}$, then there exists an incentive strategy of the leader $\alpha$ such that $u^\natural\in\mathcal{E}(\alpha)$
\end{enumerate}
\end{Lm}
The proof of this Lemma is the same as the proofs of Lemmas \ref{lm_optimal} and \ref{lm_constr}.

\begin{Th}
\begin{enumerate}
  \item If $(\alpha^*,u^*)$ is an inverse Stackelberg solution, then the profile of strategies $(u_0^*,u^*_1)$ with $u_0^*=\alpha^*(u^*_1)$ maximizes the value $J_0(u^*_0,u^*_1)$ over the set $\mathcal{B}$.
  \item If profile of strategies $(u_0^*,u^*_1)$ maximizes the value $J_0(u^*_0,u^*_1)$ over the set $\mathcal{B}$ then there exists an incentive strategy $\alpha^*$ such that $\alpha^*[u_1^*]=u_0^*$ and $(\alpha^*,u^*_1)$ is an inverse Stackelberg solution.
  \item If the function $u_i'\mapsto J_i(u_0,u_i',u_{-i})$ is quasiconcave for all $u_0$, $u_{-i}$, and $i=1,\ldots,n$, then there exists at least one inverse Stackelberg solutions.
\end{enumerate}
\end{Th}
\begin{proof}
The proof of the first two statements directly follows from Lemma \ref{lm_const_opt_many}.

Let us prove the third statement of the Theorem.
Define  $$K_i(u_1,\ldots,u_n)=\min_{u_0\in P_0}J_i(u_0,u_1,\ldots,u_i).$$ The functions $u_i'\mapsto K_i(u_i',u_{-i})$ are quasiconcave for all $u_{-i}$. Therefore there exists a profile of followers' strategies $u^\natural$ such that for all $u_i\in P_i$
$K_i(u^\natural)\geq K_i(u_i,u_{-i}^\natural) $. Hence, we have that any pair $(u_0,u^\natural)$ belongs to $\mathcal{B}$. Consequently, $\mathcal{B}$ is nonempty. Moreover, the set $\mathcal{B}$ is compact. This prove the existence of the pair $(u_0^*,u^*)$ maximizing $J_0$ over the set $\mathcal{B}$. The existence of inverse Stackelberg solution directly follows from the second statement of the Theorem.
\end{proof}
\section{Inverse Stackelberg Solution for Differential Games}\label{sec_diff}
As above we assume that player $0$ is a leader, when players $1,\ldots, n$ are followers.
The dynamics of the system is given by the equation
\begin{equation}\label{system}
  \dot{x}=f(t,x,u_0,u_1,\ldots,u_n), \ \ t\in [0,T], \ \ x\in\mathbb{R}^d, \ \ x(0)=x_0, \ \ u_i\in P_i.
\end{equation}
Player $i$ wishes to maximize the payoff
$$ \sigma_i(x(T))+\int_{0}^Tg_i(t,x,u_0,u_1,\ldots,u_n)dt.$$

The set
$$\mathcal{U}_i=\{u_i:[0,T]\rightarrow P_i\mbox{ measurable}\} $$ is the set of open-loop strategies of player $i$. As above the $n$-tuple of open-loop strategies of followers $u=(u_1,\ldots,u_n)$ is called the profile of strategies. For notational simplicity denote $$f(t,x,u_0,u)=f(t,x,u_0,u_1,\ldots,u_n),\ \  g(t,x,u_0,u)=g(t,x,u_0,u_1,\ldots,u_n). $$ Further, put $$\mathcal{U}=\times_{i=1}^n\mathcal{U}_i.$$ If $u_0\in\mathcal{U}_0$, $u=(u_1,\ldots,u_n)\in \mathcal{U}$, $(t_*,x_*)\in\ir$, then denote by $x(\cdot,t_*,x_*,u_0,u)$ the solution of initial value problem
$$\dot{x}(t)=f(t,x(t),u_0(t),u_1(t),\ldots, u_n(t)), \ \ x(t_*)=x_*. $$ Put
$$z_i(t,t_*,x_*,u_0,u)=\int_{t_*}^tg_i(t,x(t),u_0(t),u_1(t),\ldots,u_n(t))dt. $$ If $t_*=0$, $x_*=x_0$ we omit the arguments $t_*$ and $x_*$. Let $z(\cdot,t_*,x_*,u_0,u)=(z_0(\cdot,t_*,x_*,u_0,u),z_1(\cdot,t_*,x_*,u_0,u),\ldots, z_n(\cdot,t_*,x_*,u_0,u))$. We assume that the set of motions is closed  i.e. for all $(t_*,x_*)\in\ir$
\begin{multline*}
{\rm cl}\{(x(\cdot,t_*,x_*,u_0,u),z(\cdot,t_*,x_*,u_0,u)):u_0\in\mathcal{U}_0,u\in\mathcal{U}\}\\= \{(x(\cdot,t_*,x_*,u_0,u),z(\cdot,t_*,x_*,u_0,u)):u_0\in\mathcal{U}_0,u\in\mathcal{U}\}.
\end{multline*} Here ${\rm cl}$ denote closure in space of continuous functions on $[0,T]$.

We assume that the followers use the open-loop strategies $u_i\in\mathcal{U}_i$, when the leader's strategy is a nonanticipative strategy $\alpha:\mathcal{U}\rightarrow \mathcal{U}_0$. The nonanticipation property means that
$\alpha[u](\tau)=\alpha[u'](\tau) $ for any $u$ and $u'$ coinciding on $[0,\tau]$.

For $u_0\in\mathcal{U}_0$, $u\in\mathcal{U}$, $(t_*,x_*)$ define $$J_i(t_*,x_*,u_0,u)=\sigma_i(x(T,t_*,x_*,u_0,u))+z_i(T,t_*,x_*,u_0,u). $$ Further, put
$$J_i[t_*,x_*,\alpha,u]=J_i(t_*,x_*,\alpha(u),u). $$
We omit the arguments $t_*$ and $x_*$ if $t_*=0$, $x_*=x_0$.

We assume that the followers' solution concept is Nash equilibrium. Let $\mathcal{E}_d(\alpha)$ denote the set of Nash equilibria in the case when the leader plays with nonanticipating strategy $\alpha$:
$$\mathcal{E}_d(\alpha)=\{u\in\mathcal{U}:J_i[\alpha,u]\geq J_i[\alpha,u'_i,u_{-i}]\mbox{ for all }u_i'\in\mathcal{U}_i\}. $$

\begin{De} The pair consisting of nonanticipative strategy of the leader $\alpha^*$ and $u^*\in\mathcal{U}$ is an inverse Stckelberg solution of the differential game if
\begin{itemize}
  \item $u^*\in\mathcal{E}_d(\alpha^*)$
  \item $$J_0[\alpha^*,u^*]=\max_{\alpha}\max_{u\in\mathcal{E}_d(\alpha)} J_0[\alpha,u]. $$
\end{itemize}
\end{De} The proposed definition is analogous to the definition of inverse Stackelberg solution for static games.
The characterization in the differential game case is close to the characterization in the static game case also.

For a fixed profile of strategies of all players but $i$-th one $u_{-i}$ one can consider the zero-sum differential game of  player $0$ and player $i$. The lower value of this game is
$$V_i^-(t_*,x_*,u_{-i})=\min_{\alpha}\max_{u_i'\in\mathcal{U}_i}J_i[t_*,x_*,\alpha,u_i',u_{-i}]. $$

Let \begin{multline*}\mathcal{C}=\{(u_0,u)\in\mathcal{U}_0\times\mathcal{U}:\\J_i(t,x(t),u_0,u)\geq V_i^-(t,x(t),u_{-i}), \ \ x(\cdot)=x(\cdot,u_0,u),\ \ t\in [0,T]\}. \end{multline*}

\begin{Lm}\label{lm_ness} Let $\alpha$ be an incentive strategy of the leader.
If $u^\natural\in\mathcal{E}_d(\alpha)$ then $(\alpha[u^\natural],u^\natural)\in\mathcal{C}$.
\end{Lm}
\begin{proof}
We claim that
\begin{equation}\label{J_i_J_i_prim}
J_i[t,x^\natural(t),u_0^\natural,u^\natural]\geq J_i[t,x^\natural(t),\alpha[u_i',u_{-i}^\natural],u_i',u_{-i}^\natural]
\end{equation}
for any $u_i'\in\mathcal{U}_i$, $u_0^\natural=\alpha(u^*)$, $x^\natural=x(\cdot,\alpha[u^*],u^\natural)$. Assume the converse. This means that for some $u_i'$ and $\tau$
$$J_i[\tau,x^\natural(\tau),u_0^\natural,u^\natural]< J_i[\tau,x^\natural(\tau),\alpha[u_i',u_{-i}^\natural],u_i',u_{-i}^\natural]. $$
Consider the control $$u_i^\flat=\left\{
\begin{array}{cc}
  u_i^\natural(t), & t\in [0,\tau] \\
  u_i'(t), & t\in [\tau,T].
\end{array}
\right. $$
Denote $u_0^\flat=\alpha[u_i^\flat,u_{-i}^\natural]$, $x^\flat(\cdot)=x(\cdot,u_0^\flat,(u_i^\flat,u_{-i}^\natural))$. We have that $$J_i[\alpha,u_i^\flat,u_{-i}^\natural]\\=\sigma(x^\flat(T))+\int_{0}^T g(t,x^\flat(t),u_0^\flat,(u_i^\flat,u_{-i}^\natural))dt.
$$
Since  for $t\in [0,\tau]$ $u^\flat_i(t)=u^\natural_i(t)$, $u_0^\flat=u_0^\natural(t)=\alpha[u^\natural](t)$, $x^\flat(t)=x^\natural(t)$, and for $t\in [\tau,T]$ $x^\flat(t)=x(t,\tau,x^\natural(\tau),u_0^\flat,(u_i^\flat,u_{-i}^\natural))$ the following inequality holds
$$J_i[\alpha,u_i^\flat,u_{-i}^\natural]> \int_{0}^\tau g_i(t,x^\natural(t),u_0^\natural,u^\natural)dt+J[\tau,x^\natural(\tau),\alpha,u^\natural]= J[\alpha,u^\natural]. $$ This contradicts with the assumption $u^\natural\in\mathcal{E}_d(\alpha)$.

The inequality (\ref{J_i_J_i_prim}) yields the inequality $J_i[t,x^\natural(t),u_0^\natural,u^\natural]\geq V_i^-(t,x^\natural(t),u^\natural_{-i})$.
\end{proof}
\begin{Lm}\label{lm_suff_diff}
For any $(u_0^\natural,u^\natural)\in\mathcal{C}$ there exists a nonanticipative strategy of the leader $\alpha$ so that $\alpha(u^\natural)=u^\natural_0$ and $u^\natural\in\mathcal{E}_d(\alpha)$.
\end{Lm}
\begin{proof}
Let $u_i\in \mathcal{U}$, and let $\tau_i$ be the greatest time  so that $u_i=u_i^\natural$ on $[0,\tau_i]$. Denote $\xi_i=x(t,u_0^\natural,(u_i,u_{-i}^\natural))$. There exists  a nonanticipative strategy of the leader $\alpha_{\tau_i}$ such that
$$V_i(\tau_i,\xi_i,u_{-i}^\natural)=\max_{u_i}J[\tau_i,\xi_i,\alpha_{\tau_i},u_i,u_{-i}^\natural]. $$
Let $\alpha^*$ be a nonanticipative strategy of the leader so that
$$\alpha^*[u_i,u_{-i}^\natural](t)=\left\{
\begin{array}{cc}
  u_0^\natural(t), & t\in [0,\tau] \\
  \alpha_{\tau_i}(u_i), & t\in[\tau,T].
\end{array}
\right. $$
We have that $\alpha^*[u^\natural]=u_0^\natural$. Moreover, for any $u_i\in\mathcal{U}_i$
$J_i[\alpha,u_i,u_{-i}^\natural]=V_i(\tau_i,\xi_i,u_{-i}^\natural)\leq J_i[\alpha,u^\natural]$.
\end{proof}
\begin{Th}\label{th_diff_stack}
\begin{enumerate}
\item If the pair $(\alpha^*,u^*)$ is an inverse Stackelberg solution then $(u_0^*,u^*)\in\mathcal{C}$ and $(u_0^*,u^*_1)$ maximizes the value $J_0$ over the set $\mathcal{C}$ for $u_0^*=\alpha^*[u^*]$.
\item Conversely, if the pair $(u_0^*,u^*_1)$ maximizes the value $J_0$ over the set $\mathcal{C}$ then there exists an incentive strategy of the leader $\alpha^*$ such that $\alpha^*[u_1^*]=u_0^*$ and $(\alpha^*,u_1^*)$ is an incentive Stackelberg solution.
\end{enumerate}

\end{Th}
The  Theorem directly follows from Lemmas \ref{lm_ness}, \ref{lm_suff_diff}.
\section{Existence of Inverse Stackelberg Solution for Differential Game}\label{sec_ex_diff}
In this section we consider the differential game in the mixed strategies. This means that we replace the system (\ref{system}) with the control system described by the equation
\begin{equation}\label{system_mixed}
\dot{x}(t)=\int_{P_0}\int_{P_1}\ldots\int_{P_n}f(t,x(t),u_0,u_1,\ldots,u_n)\mu_n(t,du_n)\ldots\mu_1(t,du_1)\mu_0(t,du_0).
\end{equation} Here $\mu_i(t,\cdot)$ are probabilistic measures on $P_i$.
We denote the solution of initial value problem for equation (\ref{system_mixed}) and the position $(t_*,x_*)$ by $x(\cdot,t_*,x_*,\mu_0,\mu_1,\ldots,\mu_n)$. Further, let $\mathcal{M}_i$ be a set of function $\mu_i(t,du_i)$ such that for all $t$ $\mu_i(t,\cdot)$ is a probabilistic measure on $P_i$ and $t\mapsto \mu(t,\cdot)$ is weakly measurable i.e.
$$t\mapsto\int_{P_i}\phi(u_i)\mu(t,du_i) $$ is measurable for any continuous function $\varphi$.

As above we call the $n$-tuple $\mu=(\mu_1,\ldots,\mu_n)$ the profile of followers' mixed strategies. Denote the set of followers' strategies by $\mathcal{M}$. Put $x(\cdot,t_*,x_*,\mu_0,\mu)=x(\cdot,t_*,x_*,\mu_0,\mu_1,\ldots,\mu_n)$, $x(\cdot,t_*,x_*,\mu_0,\mu_i',\mu_{-i})=x(\cdot,t_*,x_*,\mu_0,(\mu_i',\mu_{-i}))$.

Further denote $$P_{-i}=\times_{j\neq i}P_j.$$ If $m=(m_1,\ldots,m_n)\in\mathcal{M}$ then denote with a slight abuse of notation $m(du)=m_1(du_1)\ldots m_n{du_n}$. Further, $$\int_P\varphi(u)m(du) $$ means the integral by the measure
$m_1(du_1)\ldots m_n(du_n)$ over the set $P=P_1\times \ldots\times P_n$.
Analogously, if $m_{-i}$ is a $(n-1)$-tuple of measures $(m_j)_{j\neq i}$ then we assume that $m_{-i}(du_{-i})\triangleq\times_{j\neq i}m_j(du_j)$. Thus,
$$\int_{P_{-i}}\varphi(u_{-i})m_{-i}(du_{-i}) $$ designates the integral by the measure $\times_{j\neq i}m_j(du_j)$ over the set $P_{-i}$.

For the given position $(t_*,x_*)\in\ir$, and measures $\mu_0\in\mathcal{M}_0$, $\mu\in\mathcal{M}$ the corresponding payoff of player $i$ is equal to
\begin{multline*}J_i(t_*,x_*,\mu_0,\mu)=
\sigma_i(x(T,t_*,x_*,\mu_0,\mu))\\+\int_{t_*}^T\int_{P_0}\int_{P} g_i(t,x(t,t_*,x_*,\mu_0,\mu),u_0,u)\mu_0(t,du_0)\mu(t,du)dt.
\end{multline*}

As above the  mapping $\alpha:\mathcal{M}\rightarrow\mathcal{M}_0$ satisfying  condition of feasibility
(the equality $\mu'$ and $\mu''$ on $[0,\tau]$ yields the equality $\alpha[\mu'](t,\cdot)=\alpha[\mu''](t,\cdot)$ on $[0,\tau]$)  is called nonanticipative strategy.

\begin{Th}\label{Th_existence} Assume that the following conditions hold true for each $i=\overline{1,n}$
\begin{enumerate}
  \item $x\mapsto\sigma_i(x)$ is concave;
  \item $g_i(t,x,u_0,u)=g_i^0(t,x,u_{-i})+g^1_i(t,u_0,u_{-i})+g^2_i(t,u)$ and the function ${x\mapsto g_i^0(t,x,u_{-i})}$ is concave.
\end{enumerate}

Then there exists an inverse Stackelberg solution in mixed strategies $(\alpha^*,\mu^*)$.
\end{Th}
\begin{proof}
Let us prove that the set $\mathcal{C}$ is nonempty.

Since the players use mixed strategies the Isaacs condition holds for each $i=\overline{1,n}$ i.e. for all profile of measures $m_{-i}$ and any vector $s\in\mathbb{R}^d$ the following equality is valid
\begin{multline*}
\min_{m_0}\max_{m_i}\int_{P_0}\int_{P_i}\int_{P_{-i}}[\langle s,f(t,x,u_0,u_1,\ldots,u_n)\rangle+g_i(t,x,u_0,u_1,\ldots,u_n)]
\\m_{-i}(du_{-i})m_i(du_i)m_0(du_0)\\=
\max_{m_i}\min_{m_0}\int_{P_0}\int_{P_i}\int_{P_{-i}}[\langle s, f(t,x,u_0,u_1,\ldots,u_n)\rangle+g_i(t,x,u_0,u_1,\ldots,u_n)]\\
m_{-i}(du_{-i})m_i(du_i)m_0(du_0).\end{multline*}
Therefore
$$V^-(t_*,x_*,\mu_{-i})=V^+(t_*,x_*,\mu_{-i}) =\max_{\beta_i}\min_{\mu_0\in\mathcal{M}_0}J_i(t_*,x_*,\mu_0,\beta_i[\mu_0],\mu_{-i}). $$
Here $\beta_i$ denotes a mapping $\mathcal{M}_0\rightarrow\mathcal{M}_i$ satisfying feasibility property.

Define the multivalued map $\mathcal{G}:\mathcal{M}_0\times\mathcal{M}\multimap \mathcal{M}_0\times\mathcal{M}$ by the rule $(\mu_0',\mu')\in \mathcal{G}(\mu_0,\mu)$ if for each $i=\overline{1,n}$ $$J_i(t,x_i(t),\mu_0',\mu_i',\mu_{-i})\geq V^-_i(t,x_i(t),\mu_{-i}).$$ Here $x_i(\cdot)=x(\cdot,\mu_0',\mu_i',\mu_{-i})$.

Note that the set $\mathcal{G}(\mu_0,\mu)$ is convex for all $\mu_0\in\mathcal{M}_0$, $\mu\in\mathcal{M}$. Moreover, $\mathcal{G}$ has a closed graph. Let us prove the nonemptiness of $\mathcal{G}(\mu_0,\mu)$.

Put $\mu_0'=\mu_0$. From Bellman principle it follows that
\begin{multline}\label{bellman}
V^+_i(t_*,x_*,\mu_{-i})  =\max_{\beta_i}\min_{\mu_0\in\mathcal{M}_0}\Bigl[V(t_+,x(t_+,t_*,x_*,\mu_0,\beta_i[\mu_{-i}],\mu_{-i})) \\+\int_{t_*}^{t^+}\int_{P_0}\int_{P_i}\int_{P_{-i}} g_i(t,x(t_+,t_*,x_*,\mu_0,\beta_i(\mu_{-i}),\mu_{-i})),u_0,u_{-i},u_{-i})\\\mu_{-i}(t,du_{-i})\beta_{i}[\mu_0](t,du_{i})\mu_0(t,du_0)dt\Bigr].
\end{multline}
Let $N$ be a natural number. Put $t^k_N=Tk/N$. Let $\beta_{i,N}^k$ maximize the right-hand side at (\ref{bellman}) for $t_*=t_N^k$, $t_+=t_N^{k+1}$, $x_*=\xi^{k-1}_{i,N}$. Here $\xi_{i,N}^k$ is defined inductively by the rule $$\xi_{i,N}^0=x_0,\ \ \xi_{i,N}^k=x(t_{i,N}^{k},t_{i,N}^{k-1},\xi_{i,N}^{k-1},\mu_0,\beta_{i,N}^{k-1}[\mu_0],\mu_{-i}).$$

Put $\tilde{\mu}_{i,N}(t,\cdot)=\beta_{i,N}^k[\mu_0](t,\cdot)$ for $t\in [t_{N}^{k-1},t_N^{k})$. Denote $x_{i,N}(\cdot)=x(\cdot,t_0,x_0,\mu_0,\tilde{\mu}_{i,N},\mu_{-i})$. Note that $\xi_{i,N}^k=x_{i,N}(t_{N}^k)$.
We have for $k<l$ the inequality
\begin{multline*}
V^+_i(t_N^k,x_{i,N}(t_N^k),\mu_{-i})\leq V^+_i(t_N^l,x_{i,N}(t_N^l),\mu_{-i})\\+\int_{t_N^k}^{t_N^l}\int_{P_0}\int_{P_i}\int_{P_{-i}}
g_i(t,x_{i,N}(t),u_0,u_i,u_{-i})\mu_{-i}(t,du_{-i})\mu_i(t,du_i)\mu_0(t,du_0)dt
\end{multline*}
Note that $V^+_i(t_N^N,\xi_{i,N}^N,\mu_{-i})=\sigma_i(\xi_{i,N}^N)$.

Using the continuity of the function $V_i^+$ we get that
\begin{multline}\label{main_ineq_N}
V^+_i(t_*,x_{i,N}(t_*),\mu_{-i})\leq V^+_i(T,x_{i,N}(T),\mu_{-i})\\+\int_{t_*}^{T}\int_{P_0}\int_{P_i}\int_{P_{-i}}
g_i(t,x_{i,N}(t),u_0,u_i,u_{-i})\mu_{-i}(t,du_{-i})\mu_i(t,du_i)\mu_0(t,du_0)dt+\delta_N.
\end{multline} Here $\delta_N\rightarrow 0$, as $N\rightarrow\infty$.

There exists a sequence $\{\tilde{\mu}_{i,N_r}\}$ converging to some $\mu'_i\in\mathcal{M}_i$, as $r\rightarrow\infty$. Therefore
$x_{i,N_r}(\cdot)=x(\cdot,t_0,x_0,\mu_0,\tilde{\mu}_{i,N_r},\mu_{-i})$  tends to $x_i(\cdot)=x(\cdot,t_0,x_0,\mu_0,{\mu}_{i}',\mu_{-i})$. This and inequality (\ref{main_ineq_N}) yield the inequality
\begin{multline*}
V^+_i(t_*,x_{i}(t_*),\mu_{-i})\leq V^+_i(T,x_{i}(T),\mu_{-i})\\+\int_{t_*}^{T}\int_{P_0}\int_{P_i}\int_{P_{-i}}
g_i(t,x_{i}(t),u_0,u_i,u_{-i})\mu_{-i}(t,du_{-i})\mu_i(t,du_i)\mu_0(t,du_0)dt.
\end{multline*}

Consider the profile of followers' strategies $\mu'=(\mu'_1,\ldots,\mu'_n)$. We have that $(\mu_0,\mu')\in\mathcal{G}(\mu_0,\mu)$. 

Since $\mathcal{M}_0\times\mathcal{M}$ is compact, and $\mathcal{G}$  is an upper semicontinuous multivalued map with nonempty convex compact values, we get that $\mathcal{G}$ admits the fixed point $(\mu_0^*,\mu^*)$. Obviously, it belongs to $\mathcal{C}$. The consequence of the Theorem follows from this and  Theorem \ref{th_diff_stack}.
\end{proof}

\section*{Acknowledgments}

The work was supported by RFBR (project N~12-01-00537), and Presidium of RAS (projects 12-P-1-1002, 12-P-1-1012).

\end{document}